\documentclass[a4paper,conference]{IEEEtran}
\usepackage{ifpdf}
\usepackage{cite}
\usepackage[pdftex]{graphicx}
\usepackage{epstopdf}
\usepackage{amsmath,amssymb,amsthm,amsfonts}
\usepackage{algorithm}
\usepackage{algorithmic}
\usepackage{array}
\usepackage{url}
\usepackage{color}

\newtheorem{theorem}{Theorem}

\newtheorem{assumption}[theorem]{Assumption}

\newtheorem{remark}[theorem]{Remark}

\newcommand{\algoref}[1]{Algorithm~\ref{#1}}

\newcommand{\hA}{\hat{A}}
\newcommand{\hB}{\hat{B}}
\newcommand{\hC}{\hat{C}}
\newcommand{\hD}{\hat{D}}

\newcommand{\blue}[1]{{\color{blue}{#1}}}

\begin{document}

\title{An Explicit Convergence Rate 
for \\ Nesterov's Method from SDP
}

\author{
\IEEEauthorblockN{Sam Safavi}
\IEEEauthorblockA{
sam.safavi@tufts.edu}
\and
\IEEEauthorblockN{Bikash Joshi}
\IEEEauthorblockA{
bikash.joshi@imag.fr}
\and
\IEEEauthorblockN{Guilherme Fran\c{c}a}
\IEEEauthorblockA{
guifranca@jhu.edu}
\and
\IEEEauthorblockN{Jos\'e Bento}
\IEEEauthorblockA{
jose.bento@bc.edu}
}

\maketitle

\begin{abstract}
The framework of Integral Quadratic Constraints (IQC) introduced by 
Lessard et al. (2014) 
reduces the computation
of upper bounds on the convergence 
rate of several optimization algorithms to 
semi-definite programming (SDP). 
In particular, 
this technique was applied to Nesterov's accelerated
method (NAM). For quadratic functions, this SDP was explicitly 
solved leading to a new bound on the convergence rate 
of NAM, and for arbitrary strongly convex functions it was shown 
numerically that IQC can improve bounds from Nesterov (2004). 
Unfortunately, an explicit analytic solution to the SDP was not provided. 
In this paper, we provide such an analytical solution,
obtaining a new general and explicit upper bound on the convergence 
rate of NAM, which we further optimize over its parameters.
To the best of our knowledge, this is the best, and explicit, upper 
bound on the convergence
rate of NAM for strongly convex functions.
\end{abstract}

\IEEEpeerreviewmaketitle

\section{Introduction}

Consider the problem
\begin{equation}
\label{eq:minimize}
\min_{x\in\mathbb{R}^p} f(x)
\end{equation}
under the following additional 
assumption, which holds throughout this paper.
\begin{assumption}\label{assumption}
\begin{enumerate}
\item The function $f$ is convex, closed and proper;
\item Let $S_d(m,L)$ be the set of 
functions $h:\mathbb{R}^d \to \mathbb{R}\cup\{+\infty\}$ such that
$
m\|x-y\|^2\le\left(\nabla h(x)-\nabla h(y)\right)^T(x-y)\le L\| x-y\|^2
$
for all $x,y\in\mathbb{R}^d$ where $0 < m \le L < \infty$
and $\| \cdot \|$ denotes the Euclidean norm;
We assume that 
$f \in S_p(m,L)$, i.e. $f$ is strongly
convex and $\nabla f$ is Lipschitz continuous.
\end{enumerate}
\end{assumption}
In this paper, we provide a new bound on the convergence rate of
NAM when solving \eqref{eq:minimize}.

NAM has wide applications in machine learning.
It is the base of the well-known FISTA algorithm
largely used to solve problems arising in signal processing
\cite{BeckTeboulle},
and it was also extensively applied in compressed sensing, as for instance
in \cite{BeckerCandes,Tropp}.
A trace norm regularization using NAM was proposed in \cite{JiYe}, 
which has applications
in multi-task learning, matrix classification and matrix completion.
Even to train deep neural networks, it was shown that NAM
with a careful initialization is able to achieve state-of-the-art
accuracy \cite{Hinton}.

NAM is parametrized by $\alpha > 0$ and $\beta \geq 0$ and takes the form in
\algoref{alg:NAM}. We assume that $\alpha$ and $\beta$ are fixed. 
A classical choice for these parameters is \cite{Nesterov}
\begin{equation}\label{eq:standard_choice_nest_param}
\alpha = 1/L, \qquad \beta = (\sqrt{\kappa} - 1)/(\sqrt{\kappa} + 1),
\end{equation}
where $\kappa = L/m$.
We define an upper bound on the convergence rate of NAM,
for fixed $\alpha$, $\beta$ and function $f$, as any $\tau \in [0,1]$
for which%
\begin{equation} \label{eq:def_of_rate_bound}
\|x_{t} - x_*\| \leq C \tau^t \|x_{0} - x_*\|,
\end{equation}
where $C > 0$ is a constant, and $x_*$ is a fixed point of \algoref{alg:NAM}.
Choosing $\alpha$ and $\beta$ according to \eqref{eq:standard_choice_nest_param},  
\cite{Nesterov} uses the technique of \emph{estimate sequences} and obtains
\begin{equation}\label{eq:bound_give_by_nesterov04}
\tau =\tau_{NG} \triangleq \sqrt{1 - 1/\sqrt{\kappa}}.
\end{equation}
In addition, if $f$ is quadratic, then 
\begin{equation}\label{eq:bound_give_by_nesterov04_for_quadratics}
\tau = \tau_{NQ} \triangleq  1 - 1/\sqrt{\kappa}.
\end{equation}
In \cite{Nesterov} it was also shown that any first order method must obey
\begin{equation}
\tau \geq \tau_{BP} \triangleq  1 - 2/\sqrt{\kappa + 1}.
\end{equation}

\begin{algorithm}[b]
\caption{Nesterov's accelerated method (parameters $\alpha$,  $\beta$)} \label{alg:NAM}
\begin{algorithmic}[1]
  \STATE Initialize $x_0, x_1$
  \REPEAT
  \STATE $y_{t} = (1  + \beta)x_t - \beta x_{t-1}$
  \STATE $x_{t+1} = y_t - \alpha \nabla f (y_t)$
  \UNTIL{stop criterion}
\end{algorithmic}
\end{algorithm}

\begin{figure}[t!]
\centering
\includegraphics[trim=0 0 0 0.5cm,clip,width=0.85\linewidth]{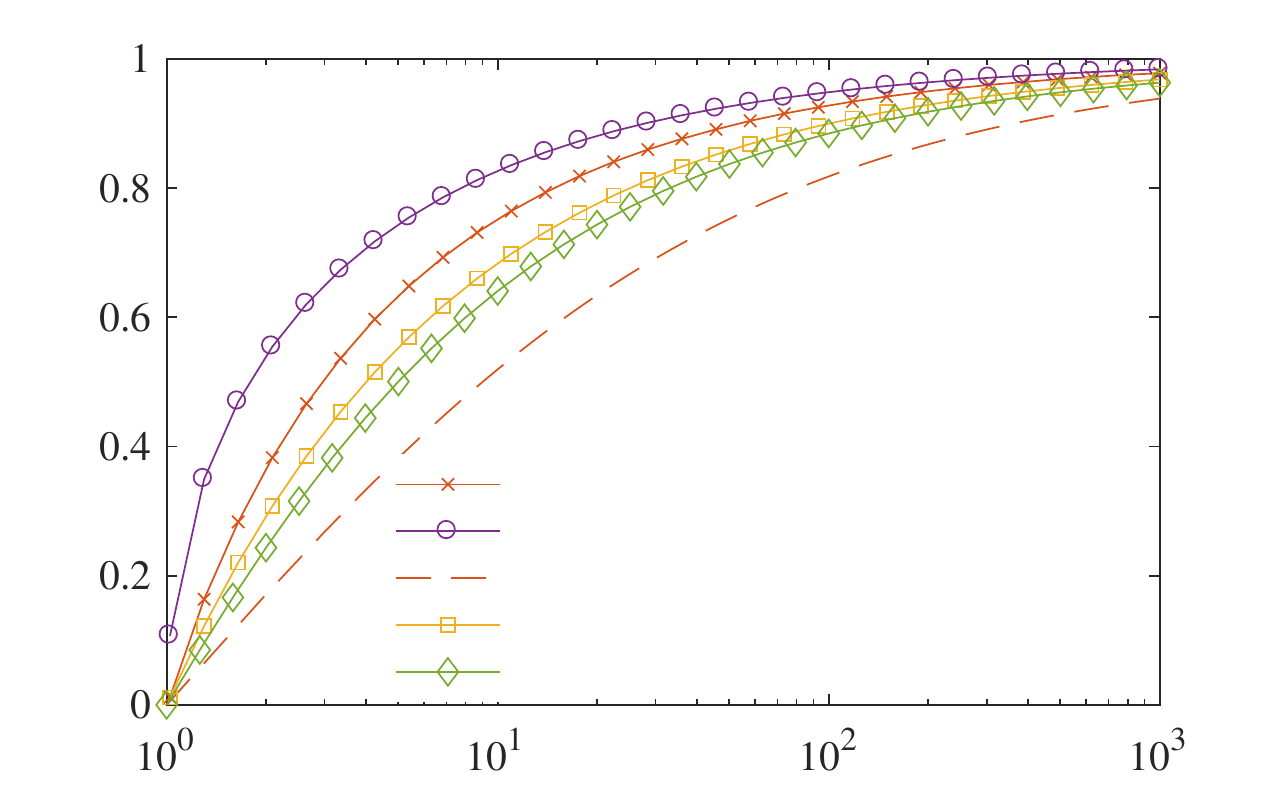}
\put(-123,5){\small $\kappa = L/m$}
\put(-130,50){\small $\tau_{LG}$ \hspace{-0.1cm}}
\put(-114,55){\tiny (Lessard et al. numerical bound for }
\put(-114,50){\tiny  general $f$, \blue{analytical expression not known})}
\put(-130,43){\small $\tau_{NG}${\tiny (Nesterov's  bound for general $f$)}}
\put(-130,35){\small $\tau_{BP}${\tiny (Best possible rate for any $1$st order method)}}
\put(-130,27){\small $\tau_{NQ}${\tiny (Nesterov's  bound for quadratics)}}
\put(-130,19){\small $\tau_{LQ}${\tiny (Lessard et al.  bound for quadratics.)}}
\caption{ Different known linear rate bounds for NAM. 
}
\label{fig:known_bound_plots}
\end{figure}
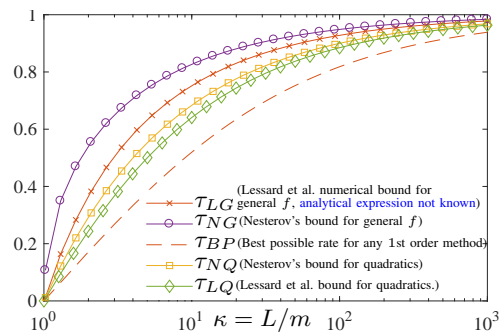

Several recent works have revisited NAM and 
computed bounds on its convergence rate based on
different techniques. Although these re-derivations have increased 
our understanding of NAM, and in some cases even inspiring new variations, 
they have not improved previous results. A partial exception is 
\cite{Lessard}, where they reduce computing a bound on the rate of convergence to finding solutions to 
a semi-definite programming (SDP) problem. 
This SDP has multiple solutions, each of which
gives a bound on the convergence rate,
some better than others. For quadratic functions,  \cite{Lessard} explicitly solve this SDP, optimize the result over $\alpha$ and $\beta$, and obtain 
a new improved bound on the convergence rate of NAM with the tuning rules
\begin{equation}\label{eq:lessard_choice_alpha_beta_for_quads}
\begin{split}
\alpha &= 4/(3L + m), \
\beta = (\sqrt{3 \kappa+1} - 2)/(\sqrt{3 \kappa+1} + 2), \\
\tau &= \tau_{LQ} \triangleq 1 - 2/\sqrt{3 \kappa +1}.
\end{split}
\end{equation}
For general strongly convex functions, they numerically solve this SDP and
obtain $\tau_{LG}$ as shown in Fig. \ref{fig:known_bound_plots}. From
the plot we see that the results from the IQC-framework improve on 
\eqref{eq:bound_give_by_nesterov04}.
However, no explicit and analytical solution to the SDP
associated to NAM was provided.
Even more discouraging is the fact that the only explicit solution
obtained was 
for Gradient Descent (GD), yielding a previously known bound on convergence rate.

On the other hand, in our recent paper at ISIT 2016, \cite{franca2016explicit}, 
we show that it is possible to extract explicit solutions from the IQC-framework for non
classical optimization algorithms and for general strongly convex functions. 
In particular, optimally tuning ADMM algorithm,
we obtain a convergence rate that matches 
the $\tau_{BP}$, the best possible for any first order methods.

The main contribution of this paper is to apply the IQC framework of \cite{Lessard}
to obtain an explicit and new bound on the convergence rate of NAM. In particular,
we derive an analytical solution to the corresponding SDP for which 
\cite{Lessard} only provides numerical solutions. 
To the best of our knowledge, our result
is the best explicit bound for NAM and arbitrary
strongly convex functions.
It is also one of the only three explicit bounds obtained from 
the IQC-framework so far; others are for GD and ADMM.

%
%
%
%
\section{Related Work}

Several recent works have revisited NAM and computed bounds on its convergence rate based on different techniques.
In addition to \cite{Lessard}, the following works are relevant.
\cite{allen2014linear} views NAM
as a linear coupling between GD and Mirror Descent, and, 
for $f \in S_p(0,L)$, re-derives the previously known 
bound 
$
f(x_t) - f(x_*) = \mathcal{O}\left(\frac{L}{t^2}\right),
$
\cite{nesterov2005smooth}, with the choice $\alpha = \tfrac{C}{\sqrt{L}}$ 
and $\beta = \tfrac{1}{\alpha L + 1}$, 
which is different from Nesterov's bound. 
This rate is not of the type 
\eqref{eq:def_of_rate_bound} that we consider in this paper. 

The work of \cite{su2016differential} views (an adaptive version of) NAM with 
$\beta = \frac{t - 1}{t + r -1}$ 
as the discretization of the second-order ODE 
$\frac{d^2 x}{d t^2} + \frac{r}{t} \tfrac{dx}{dt} + \nabla f(x) = 0$. 
For $f \in S_p(m,L)$ and $2 \leq \alpha \leq 2r/3$ they obtain
$
f(x_t) - f(x_*) \leq \frac{C(\alpha, r)}{t^\alpha}.
$
If $m > 0$, this leads to 
$\|x_t - x_*\| = \mathcal{O}\left(\frac{1}{t^{\alpha/2}}\right)$.
Unfortunately, \cite{su2016differential} show that their framework is incapable of providing linear convergence rates in general, which we know to hold
for $f \in S_p(m,L)$.

The work of \cite{flammarion2015averaging} does not give a bound for the fixed step-size NAM for a general smooth function but only for an adaptive NAM and a function $f\in S_p(0,L)$ that convex and quadratic. For such a function, and for $\beta = 1 - \frac{2}{t+1}$ and $\alpha = 1/L$, they show that $f(x_t) - f(x_*) = \mathcal{O}(\frac{L}{t^2})$.

The work of \cite{arjevani2016lower} focuses only on a convex quadratic function $f\in S_p(m,L)$, and obtain $\tau = 1 - 1/\sqrt{\kappa}$ for
$\alpha = 1/L$ and 
$\beta = (\sqrt{\kappa} - 1)/(\sqrt{\kappa} + 1)$, basically
re-deriving \eqref{eq:bound_give_by_nesterov04_for_quadratics}.

Finally, \cite{bubeck2015geometric} gives a possible geometric interpretation
of why NAM accelerates convergence. For their NAM-type method,  
the result $\tau = \sqrt{1 - 1/\sqrt{\kappa}}$ is obtained, 
basically re-deriving \eqref{eq:bound_give_by_nesterov04}.

\section{Main results}

We start recalling results from \cite{Lessard}.
\algoref{alg:NAM} can be studied through a linear dynamical system 
involving the matrices%
\footnote{The connection between \algoref{alg:NAM} and these matrices 
can only be formally established
if we use $\hat{A}\otimes I_p$, $\hat{B}\otimes I_p$ ,$\hat{C}\otimes I_p$ 
and $\hat{D}\otimes I_p$, where $\otimes$ is the Kronecker product 
and $I_p \in\mathbb{R}^{p\times p}$ is the identity matrix. 
Note, however, that Theorem \ref{lessard_theo} holds with the 
matrices $\hat{A}$, $\hat{B}$, $\hat{C}$ and $\hat{D}$ 
exactly as specified in \eqref{eq:def_of_matrices}. 
See \cite[Section 4.2]{Lessard} for more details.}
\begin{equation}\label{eq:def_of_matrices}
\begin{aligned}
\hA &= \left[
\begin{array}{ccc}
 \beta +1 & -\beta  & 0 \\
 1 & 0 & 0 \\
 L (-\beta -1) & \beta L & 0 \\
\end{array}
\right],
& \hB &= \left[
\begin{array}{c}
 -\alpha  \\
 0 \\
 1 \\
\end{array}
\right], \\
\hC &= \left[
\begin{array}{ccc}
 L (\beta +1) & -L \beta  & \rho^2  \\
 -m (\beta +1) & m \beta  & 0 \\
\end{array}
\right], 
& \hD &= \left[
\begin{array}{c}
 -1 \\
 1 \\
\end{array}
\right],
\end{aligned}
\end{equation}
inserted in a nonlinear feedback loop, where the 
feedback gain is essentially $\nabla f$. 
The constant $\rho > 0$ will be specified later.
See Figure \ref{fig:linear_system_with_non_linear_feedback} 
for an illustration.
The stability of this dynamical system is related to the convergence 
rate of \algoref{alg:NAM}, which involves numerically solving
a $4\times4$ semidefinite program.

\begin{figure}
\centering
\includegraphics[scale=0.3]{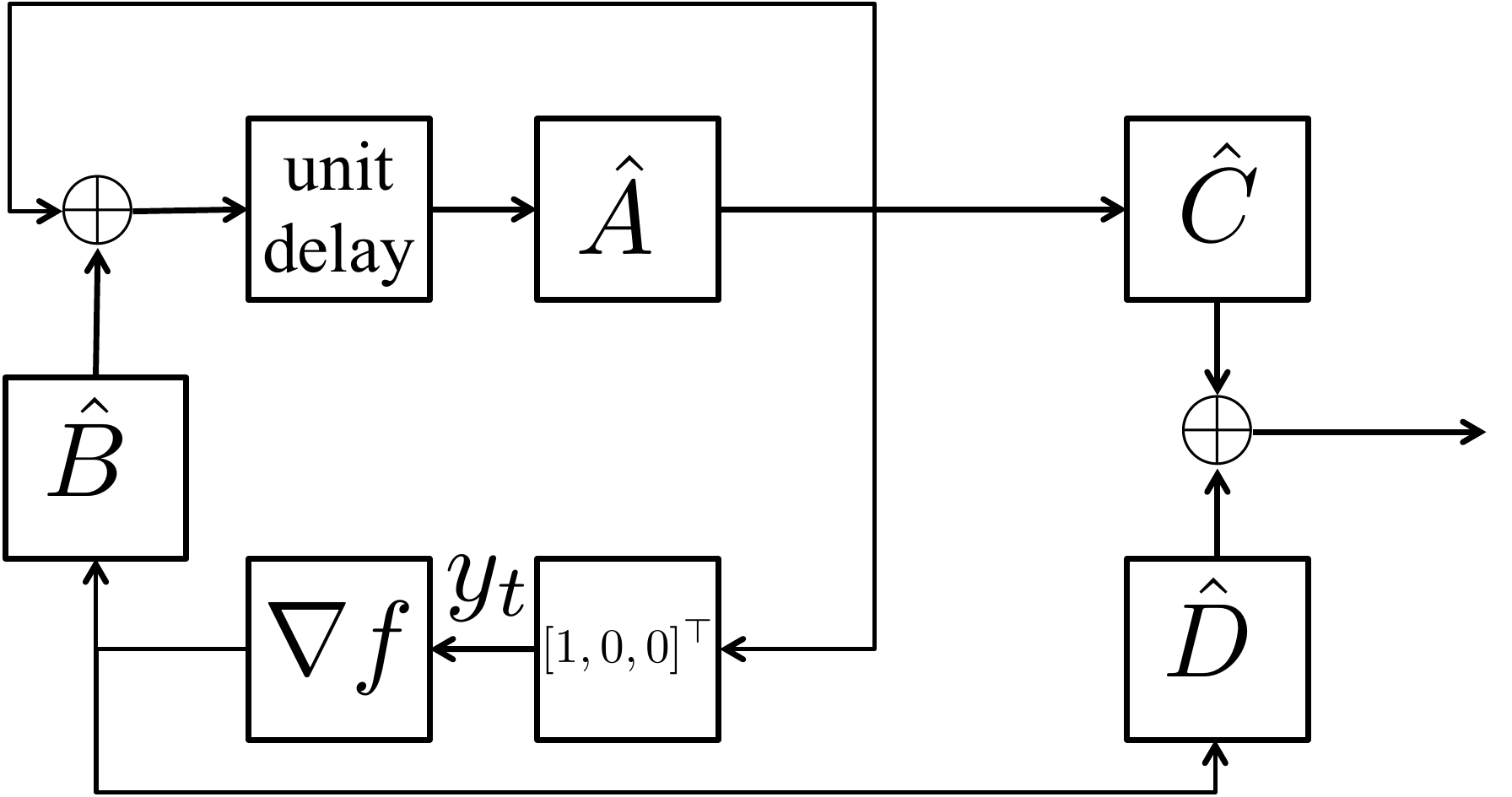}
\caption{ The variables in NAM appear in
a linear system inserted in a non-linear feedback loop. Above, we show $y_t$ only. $[1,0,0]^\top$ outputs the first component of its input. The system is more complex than NAM, and, in particular, the matrices $\hat{C}$ and $\hat{D}$ are used to probe it.
\cite{Lessard} use the properties of the output of this probe to prove properties about the convergence of NAM as stated in Theorem \ref{lessard_theo}.
}
\label{fig:linear_system_with_non_linear_feedback}
\end{figure}

\begin{theorem}[\cite{Lessard}] \label{lessard_theo}
Let 
$\left\{ x_t \right\}$
evolve according to \algoref{alg:NAM} for fixed $\alpha > 0$ and $\beta \geq 0$. 
Let $x_*$ be a fixed point of the algorithm.
Fix $0 < \rho \leq \tau < 1$.
If there exists a $3\times 3$ 
matrix $P \succ 0$ and a constant $\lambda \ge 0$
such that
\begin{equation}\label{semidefinite}
\begin{bmatrix}
\hA^T P \hA-\tau^2 P & \hA^T P \hB \\ \hB^T P \hA & \hB^T P \hB 
\end{bmatrix} \hspace{-0.1cm}+\hspace{-0.1cm}  
\lambda\begin{bmatrix} \hC  & \hD    \end{bmatrix}^{T}
\hspace{-0.1cm}  M\hspace{-0.1cm}  
\begin{bmatrix} \hC & \hD  \end{bmatrix} \preceq 0,
\end{equation}
where
$
M = \left[ \begin{smallmatrix} 0 & 1 \\ 1 & 0 \end{smallmatrix}\right]
$,
then, for all $t\ge 0$, we have 
\begin{equation}\label{bound_admm}
\| x_t - x_* \| \le \sqrt{\kappa_P} C_0 \, \tau^t ,
\end{equation}
where the constant $C_0 = \sqrt{\|x_1 - x_*\|^2 + \|x_0 - x_*\|^2}$,
and $\kappa_P = \sigma_{\max}(P)/\sigma_{\min}(P)$
is the condition number of $P$.
\end{theorem}
Note that any fixed point $x_*$ of \algoref{alg:NAM} 
satisfies the KKT conditions for problem \eqref{eq:minimize}, which due to 
strong convexity make $x_*$ the unique minimizer.
Thus, Theorem \ref{lessard_theo} enables us to find an explicit convergence
rate $\tau = \tau(\alpha,\beta,L,m)$: just find $P$, 
$\lambda$, $\rho$ and $\tau$ satisfying the conditions above.

Unfortunately, \cite{Lessard} does not give an explicit 
expression for $\tau$ as a function of $\kappa$, $\alpha$ and $\beta$. 
Our main result in this paper provides such an explicit
formula when $\alpha = 1/L$.
To arrive at this result, we first prove a series of intermediate steps.
\begin{theorem}\label{th:conditions_for_IQC_negative_definite}
Equation \eqref{semidefinite} holds if  
$\beta > 0$,
$\kappa > 1$,
$\lambda=\alpha = L = 1$, 
$\rho=\tau > 0$, $\tau$
is such that
\begin{multline}\label{eq:explicit_tau_for_NAM}
-4(-\kappa+1 )^2\beta^2(-2+\omega)+\omega(\kappa-1+\kappa\omega)^2 \\
-4(-\kappa+1)\beta\omega(-3\kappa+1+\kappa\omega)=0,
\end{multline}
where $\omega = \tau^2$ 
, and 
%
$P=\left[\begin{smallmatrix}
	a & b & c\\
	b & d & e\\
	c & e & f
\end{smallmatrix}\right],
$
where 
\begin{align}
a&= -\left(\tfrac{1}{\beta }+2\right) \omega 
+2 (\beta +2)+\tfrac{\beta  (s-1)}{\omega }-
2 (\beta+1) s,\label{eq:choice_of_a_in_P} \\
b &= \tfrac{1}{2} \left((2 \beta +1) (s-1)+\omega \right), \\
c &= \beta -\tfrac{\omega(s+\omega -1)}{2 \beta  (s-1)}
-(\beta +1) s-\omega +1,\\
d &= (1 - s)\beta,\\
e &= \omega - (1 - s)\beta,\\ 
f &= \tfrac{\omega^2}{\beta -\beta  s}\label{eq:choice_of_f_in_P}.
\end{align}
Note that $\omega$ in \eqref{eq:choice_of_a_in_P}-\eqref{eq:choice_of_f_in_P} satisfies \eqref{eq:explicit_tau_for_NAM} 
and we defined $s = \kappa^{-1}$.
\end{theorem}
\begin{remark}
Note that 
\eqref{eq:explicit_tau_for_NAM} 
is a third degree polynomial in $\omega$ with real coefficients, which
always has a real root. Moreover, all roots have a closed form expression.
This defines $\tau = \tau(\kappa, \beta)$ through $\tau^2 = \omega$.
\end{remark}
\begin{proof}[Proof of Theorem~\ref{th:conditions_for_IQC_negative_definite}]
Let 
\begin{equation}
H = \begin{bmatrix}
H_1 & H_2\\
H^\top_2 & H_3\\
\end{bmatrix}
\end{equation}
be the left hand side of 
\eqref{semidefinite} multiplied by $-1$,
where $H_1,H_2$ and $H_3$ are
$2\times2$ matrices and $H_2^\top$ denotes the
transpose of $H_2$.
To show that $H$ is positive semidefinite
we are going to use the following property of 
the Schur complement \cite{zhang2006schur}:
$H$ is positive semidefinite if and only if 
\begin{align}
H_3 &\succeq 0, \label{eq:schur1}\\
H_1 - H_2 H^\dagger_3 H^\top_2 &\succeq 0,\label{eq:schur2}\\
(I - H_3 H^\dagger_3) H^\top_2 &= 0\label{eq:schur3},
\end{align}
where $H^\dagger_3$ is the pseudoinverse of $H_3$ \cite{ben2003generalized}.

To check that conditions 
\eqref{eq:schur1}--\eqref{eq:schur3} hold, we first replace
$\lambda=\alpha = L = 1$
and formulas \eqref{eq:choice_of_a_in_P}--\eqref{eq:choice_of_f_in_P} in $H$.
Hence, for \eqref{eq:schur1}
we have
\begin{equation}
H_3 = \left[
\begin{array}{cc}
 \frac{\omega ^3}{\beta -s \beta } & -\omega  \\
 -\omega  & \frac{\beta -s \beta }{\omega } \\
\end{array}
\right]
\end{equation}
whose eigenvalues are $0$ and $\tfrac{1}{\omega}
\left(\beta(1-s) +\frac{\omega ^4}{\beta(1-s) }\right)$. 
Both are nonnegative since $s = \kappa^{-1} < 1$, $\beta >0 $ 
and $\omega > 0$.
Now we check \eqref{eq:schur3}. $H_3$ 
has no inverse but it has an explicit pseudoinverse  given by
\begin{equation}
H^\dagger_3 = \left[
\begin{array}{cc}
 -\frac{(s-1) \beta  \omega ^5}{\left(\omega ^4+(s-1)^2 \beta ^2\right)^2} & -\frac{(s-1)^2 \beta ^2 \omega ^3}{\left(\omega ^4+(s-1)^2 \beta ^2\right)^2} \\
 -\frac{(s-1)^2 \beta ^2 \omega ^3}{\left(\omega ^4+(s-1)^2 \beta ^2\right)^2} & -\frac{(s-1)^3 \beta ^3 \omega }{\left(\omega ^4+(s-1)^2 \beta ^2\right)^2}
   \\
\end{array}
\right].
\end{equation}
Replacing this expression in the left hand side of \eqref{eq:schur3} 
confirms that it holds true.
Finally, we check \eqref{eq:schur2}. After a simple, but tedious, 
calculation one can obtain 
\begin{multline}
H_1 - H_2 H^\dagger_3 H^\top_2 =\\
 \left[
\begin{array}{cc}
\frac{-4 \beta ^2 (s-1)^2 (\omega -2)-4 \beta  (s-1) \omega  (s+\omega -3)+\omega  (-s+\omega +1)^2}{4 \beta  (s-1)} & 0 \\
 0 & 0 
\end{array}
\right].
\end{multline}
Let $\delta = \delta(\omega,s,\beta)$ be the numerator of the top-left element in the matrix above.
A direct calculation shows
that $\kappa^2 \delta(\omega,\kappa^{-1},\beta)$
is the left hand side of \eqref{eq:explicit_tau_for_NAM}, 
which is zero by assumption.
Hence, $H_1 - H_2 H^\dagger_3 H^\top_2 =0$ and \eqref{eq:schur2} is true.
\end{proof}

Let $\tau = \tau( \kappa,\beta)$ be the smallest (real)
solution of \eqref{eq:explicit_tau_for_NAM} such that $\tau \in (0,1)$. Our next theorem gives an expression for the
choice of $\beta = \beta(\kappa)$ that minimizes
$\tau( \kappa,\beta)$ for each $\kappa > 1$.
\begin{theorem}\label{th:opt_alpha_beta_for_poly}
Let $\beta(\kappa)$ minimize $\tau(\kappa,\beta)$,
for fixed $\kappa > 1$. We have
\begin{align}
\beta(\kappa)&=\frac{2 \kappa -\sqrt{2 \kappa -1}-1}{2 \left(\kappa +\sqrt{2 \kappa -1}\right)},\label{eq:opt_beta}\\
\tau(\kappa,\beta(\kappa)) &= \sqrt{1-\frac{\sqrt{2 \kappa -1}}{\kappa }}\label{eq:opt_tau}.
\end{align}
\end{theorem}
\begin{proof}
Note that \eqref{eq:explicit_tau_for_NAM} is a
quadratic polynomial in $\beta$.
Its zeros are
\begin{equation}
\beta = \frac{x \pm \sqrt{y}}{z},
\end{equation}
where 
\begin{align}
x &= (\kappa -1) \omega  (\kappa  (\omega -3)+1),\\
y & =2 (\kappa -1)^2 (\omega -1) \omega  \left(\kappa  \left(\kappa  (\omega -1)^2-2\right)+1\right),\\
z &= 2 (\kappa -1)^2 (\omega -2).
\end{align}

For each $\kappa > 1$, we want to find the smallest
$\tau \in (0,1)$ for which we still have real roots in the above equation. This is the same as finding
the smallest $\omega \in (0,1)$ for which
$\left(\kappa  \left(\kappa  (\omega -1)^2-2\right)+1\right)$, a quadratic function of $\omega$,
is nonnegative. This is easy to find, yielding
\begin{equation}
\omega =1-\frac{\sqrt{2 \kappa -1}}{\kappa },
\end{equation}
for which we have
\begin{equation}
\beta = \dfrac{x}{y} = 
\dfrac{ \omega  (\kappa  (\omega -3)+1)}{ 2 (\kappa -1)(\omega -2)} 
= -\frac{-2 \kappa +\sqrt{2 \kappa -1}+1}{2 \left(\kappa +\sqrt{2 \kappa -1}\right)}.
\end{equation}
\end{proof}

\begin{theorem}\label{th:P_is_pos_def}
If $\tau$ and $\beta$ are chosen as 
\eqref{eq:opt_beta} and \eqref{eq:opt_tau}, respectively,  
and the entries in $P$ according to 
\eqref{eq:choice_of_a_in_P}--\eqref{eq:choice_of_f_in_P}, then $P \succ 0$.
\end{theorem}
\begin{proof}
Let $P'$ be $P$ with its rows and columns permuted
such that the first, second and last row/column become the
last, second and first row/column.
Note that $P'$ and $P$ have the same spectrum.
We are going to show that all the principal minors of $P'$
are strictly positive, a necessary and sufficient condition 
for positive definitiveness known as Sylvester's 
criterion \cite{meyer2000matrix}.

Replacing \eqref{eq:choice_of_a_in_P}--\eqref{eq:choice_of_f_in_P} in $P'$, the
first minor is given by
\begin{equation}
f = \frac{\kappa  \omega ^2}{\beta  (\kappa -1)} > 0.
\end{equation}
The second minor is 
\begin{equation}
\left| \begin{matrix}d & e\\ e & f\end{matrix}\right| = \frac{\beta  (\kappa -1) (\beta  (-\kappa )+\beta +2 \kappa  \omega )}{\kappa ^2},
\end{equation}
whose sign is dictated by $\beta  (-\kappa )+\beta +2 \kappa  \omega $ and
which, by substituting \eqref{eq:opt_beta}--\eqref{eq:opt_tau}, becomes
\begin{equation}
\beta  (-\kappa )+\beta +2 \kappa  \omega  = \frac{(\kappa -1) \left(2 \kappa +\sqrt{2 \kappa -1}-3\right)}{2 \left(\kappa +\sqrt{2 \kappa -1}\right)} > 0,
\end{equation}
since $\kappa \geq 1$.

The third minor is just the determinant of $P'$, which is
\begin{multline}
\tfrac{1}{\omega}\beta ^3 \left(\tfrac{1}{\kappa }-1\right)^3 (\omega -1) 
+\beta ^2 \left(\tfrac{1}{\kappa }-1\right)^2 \left(\tfrac{1}{\kappa }+3
\omega -5\right) \\
+2 \beta  \left(\tfrac{1}{\kappa }-1\right) \omega  \left(\tfrac{1}{\kappa
}+\omega -3\right)-\tfrac{1}{2} \omega 
   \left(-\tfrac{1}{\kappa }+\omega +1\right)^2.
\end{multline}
We can use \eqref{eq:explicit_tau_for_NAM} to simplify this expression to 
\begin{equation}
\frac{\beta ^2 (\kappa -1)^2 ((\omega -1) (\beta  (-\kappa )+\beta +\kappa  \omega )+\omega )}{\kappa ^3 \omega },
\end{equation}
whose sign is dictated by $(\omega -1) (\beta  (-\kappa )+\beta +\kappa  \omega )+\omega $.
If we substitute \eqref{eq:opt_beta}--\eqref{eq:opt_tau} we obtain
\begin{equation}
\begin{split}
(\omega -1) (\beta  (-\kappa )+\beta +\kappa  \omega )+\omega
&=\frac{(\kappa -1) \left(\sqrt{2 \kappa -1}-1\right)}{2 \kappa  \left(\kappa
+\sqrt{2 \kappa -1}\right)} \\
& > 0
\end{split}
\end{equation}
since $\kappa > 1$.
\end{proof}

We now provide our main result, which directly follows from 
our previous theorems and a simple rescaling argument.
\begin{theorem}
Let $f \in S_p(m,L)$ and $\kappa = L/m\geq 1$. 
Consider \algoref{alg:NAM} to solve the optimization
problem \eqref{eq:minimize}.
If $\alpha = \tfrac{1}{L}$ and $\beta = \frac{2 \kappa -\sqrt{2 \kappa -1}-1}{2
\left(\kappa +\sqrt{2 \kappa -1}\right)}$,
then
\begin{equation}
\| x_t - x_* \| \le C_0 C_1 \, \tau^t ,
\end{equation}
where $C_0 = \sqrt{\|x_1 - x_*\|^2 + \|x_0 - x_*\|^2}$,
 $C_1 > 0$ is a function of $\kappa$,
 and
 \begin{equation}\label{eq:opt_tau_in_main_theorem}
\tau  = \sqrt{1-\frac{\sqrt{2 \kappa -1}}{\kappa }}.
  \end{equation}
\end{theorem}
\begin{proof}
We can assume, without loss of generality, that $\kappa > 1$.
The case $\kappa = 1$ follows by a continuity argument,
applying a small quadratic perturbation to $f$ and letting the perturbation
converge to zero.

The convergence rate of \algoref{alg:NAM} on $f$ with $\alpha = 1/L$
is the same as its convergence rate on $\hat{f} = f/L \in S_p(m,1)$ with $\alpha = 1$.
In this setting, Theorem \ref{th:conditions_for_IQC_negative_definite} and
Theorem \ref{th:P_is_pos_def} tell us that the conditions
to apply Theorem \ref{lessard_theo} hold for our choice of $\alpha$
and $\beta$. Furthermore, according to Theorem \ref{th:opt_alpha_beta_for_poly}, for this choice
of $\alpha$ and $\beta$, the convergence rate $\tau$ satisfies \eqref{eq:opt_tau_in_main_theorem}.
\end{proof}

\section{The pathway towards the proof}

The reader might have noticed that our previous proofs
amount to substituting expressions into conditions
and subsequently checking that these conditions are satisfied.
It is enlightening to explain how we obtained these expressions 
in the first place.
Specifically, how did we obtain 
\eqref{eq:choice_of_a_in_P}--\eqref{eq:choice_of_f_in_P}
from which all other formulas follow?
In a nutshell, we built our ansatz based on numerical experimentation. Reveling
this path might be useful for other researchers to use the IQC 
framework to derive explicit formulas for other algorithms as well. 

First, we reduce the number of variables in the problem
by setting $\lambda = 1$, $\rho = \tau$ and $\alpha = L =1$.

Second, we fix $\beta>0$ and $\kappa\in(0,1)$, and use a convex optimization solver to numerically find the smallest $\tau$ for which \eqref{semidefinite} is satisfied
under the assumption that $P \succ 0$.
Let $H$ be the right hand side of \eqref{semidefinite}
multiplied by $-1$.
To find this $\tau$, we start with $\tau = 0.5$ and check if the SDP
\begin{equation}
\min_P \;\; 1 \qquad \mbox{s.t.\;\
$H \succeq 0$ and $P \succeq 0$}
\end{equation}
has a feasible solution%
\footnote{Note that the standard formulation of convex 
optimization problems, and existing solvers, does not allow us to enforce $P \succ 0$. This is why we enforce $P \succeq 0$ and later check if $P \succ 0$.}.
In the affirmative case, we reduce $\tau$, otherwise we increase $\tau$.
Notice that the eigenvalues of $H$ increase monotonically
with $\tau$. Hence, we can use bisections to find the
smallest possible $\tau$ in a few steps.
After this procedure is done, we check if $P \succ 0$. If this does not
hold, we try a different $\beta$ and/or $\kappa$.
 
Third, we repeat this procedure for several pairs of $(\beta, \kappa)$.
For each pair, we obtain numerical values for $P$ and $H$ such 
that $H \succeq 0$ and  $P \succ 0$ hold.
From these numerical values, we try to identify 
some very simple properties that $H$ or $P$ might satisfy
for all tested values of $\beta$ and $\kappa$.
Labeling the entries of $P$ as in 
Theorem~\ref{th:conditions_for_IQC_negative_definite},
the properties that we can easily guess based on our numerical experiments
are the following:
\begin{enumerate}
\item Recall that $P = P^\top = \left[ \begin{smallmatrix} 
a & b & c \\   b & d & e \\  c& e & f
\end{smallmatrix}\right]$. Then,
\begin{align}
e &= \omega - d,\label{eq:1_anzat} \\
d &= \beta (1 - m)\label{eq:2_anzat}.
\end{align}
\item Let $\Delta_i$ be the principal minor of $H$
obtained by removing the $i$th row and column.
We observe $\Delta_i = 0$ for $i = 1,\dots,4$;
\item Let $\Delta_{1,2;1,2}$ be the principal minor of
$H$ obtained from removing the $1$st and $2$nd
column/row from $H$. We observe that $\Delta_{1,2;1,2}=0$.
\end{enumerate}

Fourth, we replace \eqref{eq:1_anzat} and \eqref{eq:2_anzat}
into $H$ and we solve the condition $\Delta_1 = 0$ for $a$. This leads to 
\begin{equation}\label{eq:anazt_for_a}
a = \tfrac{1}{\omega}\left(-\beta +2 c \omega -f \omega +\beta  m+2
\omega\right).
\end{equation}
We substitute this expression into $H$ and
solve $\Delta_3=0$ for $b$, yielding 
\begin{equation}
b= \tfrac{1}{2} \left((2 \beta +1) (m-1)+\omega \right).
\end{equation}
Again, we substitute this expression in $H$ and
solve $\Delta_4 = 0$ for  $c$, obtaining
\begin{equation}\label{eq:anazt_for_c}
c = \tfrac{1}{z}\left(x \pm \sqrt{y}\right),
\end{equation}
where 
\begin{align}
x &= -2 \beta  (m-1) \omega  ((\beta +1) (m-1)-f)\nonumber\\
&\qquad \qquad -(2 \beta +1) (m-1) \omega ^2+\omega ^3,\\
y &= \omega  (-4 \beta ^2 (m-1)^2 (\omega -2)
-4 \beta  (m-1) \omega  (m+\omega -3)\nonumber\\
&\qquad \qquad +\omega  (-m+\omega +1)^2) (\beta  f
   (m-1)+\omega ^2),\\
z &= 2 \beta  (m-1) \omega.
\end{align}
We substitute the expression for $c$ with $+$ sign
in $H$ and solve $\Delta_{1,2;1,2}=0$ for $f$.
This leads to 
\begin{equation}\label{eq:anazt_for_f}
f = -\frac{\omega ^2}{\beta  (m-1)}.
\end{equation}

Finally, we eliminate $f$, $d$ and $c$ from equations
\eqref{eq:1_anzat}, \eqref{eq:anazt_for_a} and \eqref{eq:anazt_for_c}. This
leads to \eqref{eq:choice_of_a_in_P}--\eqref{eq:choice_of_f_in_P}, 
observing that $m = s = \kappa^{-1}$ when $L = 1$. 
Note that \eqref{eq:explicit_tau_for_NAM} can 
be obtained from \eqref{eq:choice_of_a_in_P}--\eqref{eq:choice_of_f_in_P} by
forcing $H \succeq 0$ (see the proof of Theorem~\ref{th:conditions_for_IQC_negative_definite}).

\section{Numerical Results and Discussion}

We first note that our optimal choice for $\beta$
in \eqref{eq:opt_beta} is numerically very close, but not equal, 
to Nestervo's choice
in \eqref{eq:standard_choice_nest_param}; see
Figure~\ref{fig:opt_beta_tau_standard} (left).
Our convergence rate for NAM is almost indistinguishable to
$\tau_{LG}$ 
in Figure~\ref{fig:known_bound_plots}, and it is 
indistinguishable from the curve obtained by 
running the Matlab code of \cite{Lessard} for the
plot of $\tau_{LG}$ with our optimal choice of $\alpha$ and $\beta$.
However, plotting $\tau_{LG}$ for the 
choice in \eqref{eq:lessard_choice_alpha_beta_for_quads} gives 
a numerical rate that is better than the one derived in this paper; 
see Figure \ref{fig:opt_beta_tau_standard} (right).
This shows that we have not extracted the best possible convergence
rate for NAM from the IQC framework. Indeed, we assumed 
that $\rho = \tau$ and $\alpha = 1/L$ which might be suboptimal. 
We did so because we were unable to find an ansatz without 
restricting $\alpha$ or $\rho$. There are too many free variables 
to perform closed form calculations, e.g. could not solve some of
the resulting polynomial equations. 

We know that any bound produced by the IQC-framework must
be above or equal to $\tau_{LQ}$ in Figure \ref{fig:opt_beta_tau_standard}.
It is an important open question to know what is the best possible bound
that the IQC-framework can produce. Can it reach $\tau_{LQ}$?

\section{Conclusion and Future Work}

We have derived a new, improved, and explicit convergence rate 
of Nesterov's accelerated method for strongly convex functions.
Our numerical experiments using the IQC framework \cite{Lessard}
show that our results can be further improved. 
Future work should include deriving better and explicit convergence rates
using the IQC framework, and demonstrating that these cannot be improved. 
It would also be important to know if 
IQC allows us to prove the best possible upper bound on the convergence
rate of Nesterov's method. To do so, one would have to produce a 
family of ``bad'' functions for which the convergence rate of 
Nesterov's method matches the rate obtained from IQC.

\section*{Acknowledgment}
This work was partially funded by 
NIH/1U01AI124302 and
NSF/IIS-1741129.


\begin{figure}
\centering
\includegraphics[trim=0 0 0 0.cm,clip,width=0.49\linewidth]{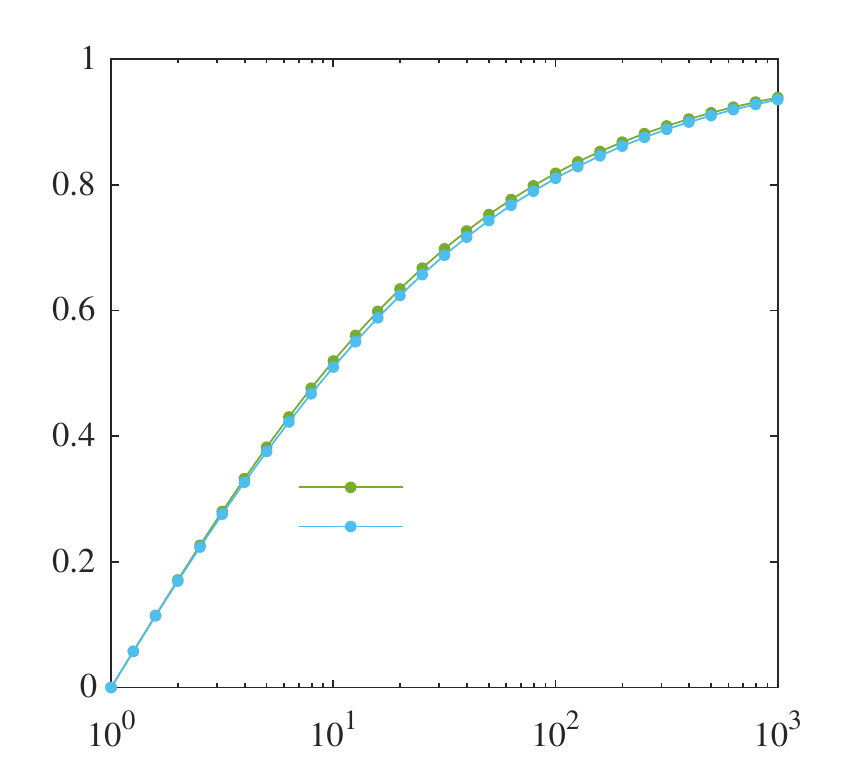}
\put(-60,5){$\kappa$}
\put(-60, 40){\small $\beta = $ eq. \eqref{eq:standard_choice_nest_param}}
\put(-60, 30){\small $\beta = $ eq. \eqref{eq:opt_beta}}
\includegraphics[trim=0 0 0 0cm,clip,width=0.49\linewidth]{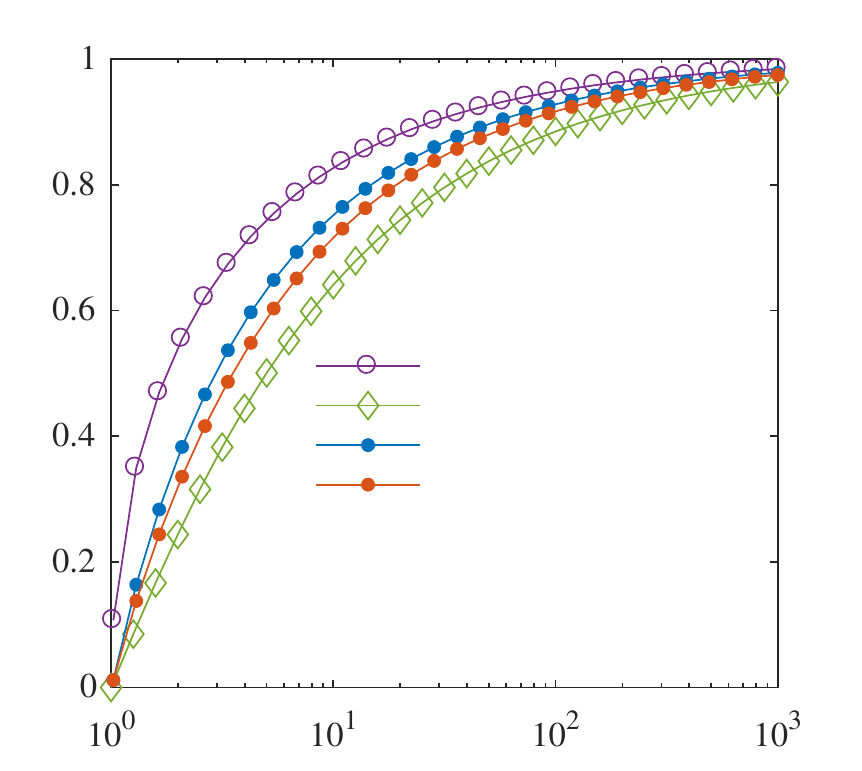}
\put(-60,5){$\kappa$}
\put(-60, 57){\small $\tau_{NG}$}
\put(-60, 50){\small $\tau_{LQ}$}
\put(-60, 43){\small $\tau = $ eq. \eqref{eq:opt_tau}}
\put(-60, 36){\small $\tau_{LG}$ with }
\put(-60, 26){\small $\alpha,\beta$  as in}
\put(-60, 16){\small in eq. \eqref{eq:lessard_choice_alpha_beta_for_quads}}
\caption{\emph{Left:} There is a very small difference 
between the standard choice for $\beta$ given 
in \eqref{eq:opt_beta} and our optimal choice 
of $\beta$ in \eqref{eq:standard_choice_nest_param}.
\emph{Right:} It is possible 
to obtain better rates than the one we derived in this paper 
if we choose $\alpha$ and $\beta$ as 
in \eqref{eq:lessard_choice_alpha_beta_for_quads}.
}
\label{fig:opt_beta_tau_standard}
\end{figure}
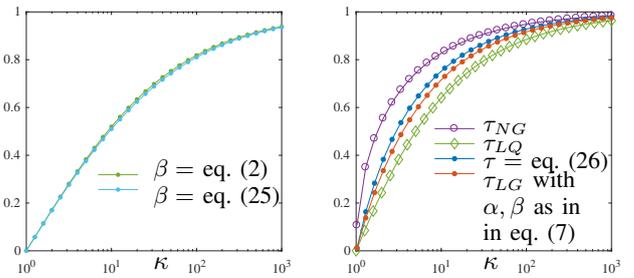
%




\bibliographystyle{IEEEtran}
\bibliography{biblio}

\begin{thebibliography}{10}
\providecommand{\url}[1]{#1}
\csname url@samestyle\endcsname
\providecommand{\newblock}{\relax}
\providecommand{\bibinfo}[2]{#2}
\providecommand{\BIBentrySTDinterwordspacing}{\spaceskip=0pt\relax}
\providecommand{\BIBentryALTinterwordstretchfactor}{4}
\providecommand{\BIBentryALTinterwordspacing}{\spaceskip=\fontdimen2\font plus
\BIBentryALTinterwordstretchfactor\fontdimen3\font minus
  \fontdimen4\font\relax}
\providecommand{\BIBforeignlanguage}[2]{{%
\expandafter\ifx\csname l@#1\endcsname\relax
\typeout{** WARNING: IEEEtran.bst: No hyphenation pattern has been}%
\typeout{** loaded for the language `#1'. Using the pattern for}%
\typeout{** the default language instead.}%
\else
\language=\csname l@#1\endcsname
\fi
#2}}
\providecommand{\BIBdecl}{\relax}
\BIBdecl

\bibitem{BeckTeboulle}
A.~Beck and M.~Teboulle, ``{A Fast Iterative Shrinkage-Thresholding Algorithm
  for Linear Inverse Problems},'' \emph{{SIAM. J. Imaging Sciences}}, vol.~2,
  pp. 183--202, 2009.

\bibitem{BeckerCandes}
S.~Becker, J.~Bobin, and E.~J. Cand{\`{e}}s, ``{NESTA: A Fast and Accurate
  First-Order Method for Sparse Recovery},'' \emph{{SIAM J. Imaging Sci.}},
  vol.~4, pp. 1--39, 2011.

\bibitem{Tropp}
J.~A. Tropp, J.~N. Laska, and M.~F. Duarte, ``{Beyond Nyquist: Efficient
  Sampling of Sparse Bandlimited Signals},'' \emph{{IEEE Transactions on
  Information Theory}}, vol.~56, pp. 520--544, 2010.

\bibitem{JiYe}
``{An accelerated gradient method for trace norm minimization}.''

\bibitem{Hinton}
I.~Sutskever, J.~Martens, G.~Dahl, and G.~Hinton, ``{On the importance of
  initialization and momentum in deep learning},'' \emph{International
  Conference on Machine Learning}, vol.~28, no.~3, pp. 1139--1147, 2013.

\bibitem{Nesterov}
Y.~Nesterov, \emph{Introductory Lectures on Convex Optimization: A Basic
  Course}.\hskip 1em plus 0.5em minus 0.4em\relax Springer, 2013.

\bibitem{Lessard}
L.~Lessard, B.~Recht, and A.~Packard, ``{Analysis and Design of Optimization
  Algorithms via Integral Quadratic Constraints},'' vol.~26, no.~1, pp. 57--95,
  2016.

\bibitem{franca2016explicit}
G.~Fran{\c{c}}a and J.~Bento, ``{An Explicit Rate Bound for Over-Relaxed
  ADMM},'' \emph{{{IEEE} International Symposium on Information Theory
  (ISIT)}}, pp. 2104--2108, 2016.

\bibitem{allen2014linear}
Z.~Allen-Zhu and L.~Orecchia, ``{Linear Coupling: An Ultimate Unification of
  Gradient and Mirror Descent},'' 2014.

\bibitem{nesterov2005smooth}
Y.~Nesterov, ``Smooth minimization of non-smooth functions,''
  \emph{Mathematical Programming}, vol. 103, pp. 127--152, 2005.

\bibitem{su2016differential}
W.~Su, S.~Boyd, and E.~Cand{\`{e}}s, ``{A differential equation for modeling
  Nesterov's accelerated gradient method: theory and insights},'' \emph{Journal
  of Machine Learning Research}, vol.~17, no. 153, pp. 1--43, 2016.

\bibitem{flammarion2015averaging}
N.~Flammarion and F.~Bach, ``{From Averaging to Acceleration, There is Only a
  Step-size},'' \emph{Conference on Learning Theory}, vol.~40, 2015.

\bibitem{arjevani2016lower}
Y.~Arjevani, S.~Shalev-Shwartz, and O.~Shamir, ``{On lower and upper bounds in
  smooth and strongly convex optimization},'' \emph{Journal of Machine Learning
  Research}, vol.~17, no. 126, pp. 1--51, 2016.

\bibitem{bubeck2015geometric}
S.~Bubeck, Y.~T. Lee, and M.~Singh, ``{A geometric alternative to Nesterov's
  accelerated gradient descent},'' 2015.

\bibitem{zhang2006schur}
F.~Zhang, \emph{The Schur complement and its applications}.\hskip 1em plus
  0.5em minus 0.4em\relax Springer Science \& Business Media, 2006, vol.~4.

\bibitem{ben2003generalized}
A.~Ben-Israel and T.~N. Greville, \emph{Generalized inverses: theory and
  applications}.\hskip 1em plus 0.5em minus 0.4em\relax Springer Science \&
  Business Media, 2003, vol.~15.

\bibitem{meyer2000matrix}
C.~D. Meyer, \emph{Matrix analysis and applied linear algebra}.\hskip 1em plus
  0.5em minus 0.4em\relax Siam, 2000, vol.~2.

\end{thebibliography}

%

\end{document}